\documentclass[11pt]{article}
\usepackage{amssymb,amsfonts,amsmath,amsthm, cite}
\usepackage{enumerate}
\usepackage[english]{babel}
\usepackage[cp1250]{inputenc}
\usepackage{algorithm2e}
\usepackage{longtable}
\usepackage{graphicx}
\usepackage{index}
\usepackage{fancyhdr}
\usepackage{lhelp}
\usepackage{optparams}
\usepackage{psfrag}
\usepackage{forloop}
\usepackage{setspace}
\usepackage{tikz}
\usepackage{subcaption}
\usepackage{pgffor}
\usepackage{xifthen}

\definecolor{lgray}{gray}{0.75}
\usetikzlibrary{decorations.pathreplacing,automata,calc,positioning}
\renewcommand{\qed}{\hfill $\square$ \bigskip}

\newcommand{\mptt}[1]{}

\newtheorem{theorem}{Theorem}
\newtheorem{corollary}[theorem]{Corollary}
\newtheorem{lemma}[theorem]{Lemma}

\newtheorem{proposition}[theorem]{Proposition}

\begin{document}

\title{\textbf{Packing chromatic numbers of finite super subdivisions of
graphs}}
\author{Rachid Lemdani $^{a,b}$ \and Moncef Abbas $^{a}$ \and Jasmina Ferme $%
^{c,d,e}$ }
\date{}
\maketitle

\begin{abstract}
The \textit{packing chromatic number} of a graph $G$, denoted by $%
\chi_\rho(G)$, is the smallest integer $k$ such that the vertex set of $G$
can be partitioned into sets $V_i$, $i\in \{1,\ldots,k\}$, where each $V_i$
is an $i$-packing. In this paper, we present some general properties of
packing chromatic numbers of \textit{finite super subdivisions} of graphs.
We determine the packing chromatic numbers of the finite super subdivisions
of complete graphs, cycles and \textit{neighborhood corona graphs} of a
cycle and a path respectively of a complete graph and a path. 
%At the end of this paper, we present the exact values of the packing chromatic numbers of some corona and neighborhood corona graphs.

\end{abstract}

\begin{center}
$^{a}$ \medskip Laboratoire AMCD\&RO, Université des Sciences et de la
Technologie, Houari Boumediene USTHB, BP32, 16111, Bab Ezzouar, Alger,
Algeria\ \\[0pt]
\medskip

$^{b}$ Université de Médéa, Faculté des Sciences, Département de Mathé%
matiques et Informatique, 26000, Médéa, Algeria \medskip\ \\[0pt]
\medskip

$^c$ Faculty of Education, University of Maribor, Slovenia\\[0pt]
\medskip

$^{d}$ Faculty of Natural Sciences and Mathematics, University of Maribor,
Slovenia\\[0pt]
\medskip

$^e$ Institute of Mathematics, Physics and Mechanics, Ljubljana, Slovenia\\[%
0pt]
\end{center}

\medskip\noindent \textbf{AMS Subj.\ Class: 05C15, 05C70, 05C12.}

%%%%%%%%%%%%%%%%%%%%%%%%%%%%%%%%%%%%%%%%%%%%%%%%%%%%%%%%%%%%%%%%
%%%%%%%%%%%%%%%%%%%%%%%%%%%%%%%%%%%%%%%%%%%%%%%%%%%%%%%%%%%%%%%%

\section{Introduction}

%%%%%%%%%%%%%%%%%%%%%%%%%%%%%%%%%%%%%%%%%%%%%%%%%%%%%%%%%

The concept of the packing chromatic number was introduced in 2008 by
Goddard et al.~\cite{goddard-2008}. First, it was presented under the name
broadcast chromatic number, and the current name was given in~\cite{bkr-2007}%
. The concept arose from the area of frequency assignment in wireless
networks \cite{FG, WR} and also has several additional applications, such as
in resource replacement and biological diversity \cite{bkr-2007}.

The packing chromatic number has been investigated in a number of papers,
for example, there exist more than 10 papers, which were published only in
the last two years (see \cite{balogh-2018, balogh-2019, bf-2018a, gt-2019,
bf-2018b, bkrw-2018, klmp-2018, kr-2019, kv-2018,kv-2019, ls-2018}). This
confirms a wide interest given to this concept. One of the main areas of
investigation has been to determine the packing chromatic numbers of
infinite graphs such as infinite grids, lattices, distance graphs, etc.~\cite%
{bar-2017, bkr-2007, ekstein-2014, finbow-2010, korze-2014, fiala-2009}. For
instance, in the last paper in a series Barnaby et al.~\cite{bar-2017} prove
that the packing chromatic number of the infinite square lattice is between
13 and 15. A lot of attention has been also given to the question of
boundedness of the packing chromatic numbers in the class of cubic graphs.
The question was answered in the negative by Balogh, Kostochka and Liu~\cite%
{balogh-2018} and there is also known an explicit construction an infinite
family of subcubic graphs with unbounded packing chromatic number (see~\cite%
{bf-2018b}). Note that the problem of determining the packing chromatic
number is computationally (very) hard \cite{FG} as its decision version is
NP-complete even when restricted to trees (see also a more recent
investigation~\cite{klmp-2018}).

It is known that the packing chromatic number satisfies the hereditary
property in the sense that a graph cannot have smaller packing chromatic
number that its subgraphs. The behaviour of the invariant under some local
operations, as edge-contraction, vertex-deletion, edge-deletion, and edge
subdivision was investigated in~\cite{bkrw-2017}, while the packing
chromatic number of subdivision of a given graph (a graph obtained from a
given graph $G$ by subdividing every edge of $G$, denoted by $S(G)$) was
considered, for example in \cite{bkr-2007}. In particular, there was proven
that for any connected graph $G$ with at least three vertices, we have: $%
\omega (G)+1\leq \chi _{\rho }(S(G))\leq \chi _{\rho }(G)+1$. Beside the
mentioned graphs operations, there is known also a (finite) super
subdivision of a graph (see e.g. \cite{kaladevi, nagarajan, mahe, WR2}), but
there is known only a little about the packing chromatic numbers of super
subdivision graphs. Actually, we have found only one paper considering the
packing chromatic number of such graphs (written by William and Roy \cite%
{WR2}). Moreover, the authors of the mentioned paper determined only the
packing chromatic number of finite super subdivision of star graphs, and
hence, we consider this topic.

Our paper is organized as follows. In the next section, we establish the
notation and define the concepts used throughout the paper. Next, we give a
lower and an upper bound for $\chi _{\rho }(FSSD_{m}(G))$, where $G$ is an
arbitrary connected graph with at least three vertices and $m$ is any
positive integer. Namely, we prove that $\omega (G)+1\leq \chi _{\rho
}(FSSD_{m}(G))\leq \chi _{\rho }(G)+1$, which generalizes the result for
subdivision of graph from \cite{bkr-2007}. Then, we prove that the packing
chromatic numbers of graphs $FSSD_{m}(G)$, $FSSD_{m+1}(G)$, $FSSD_{m+2}(G)$,
\ldots ~ are equal, when $m$ is large enough, i.e. it is greater than $\frac{%
\chi _{\rho }(G)}{\delta (G)}$. In this section we determine also the
packing chromatic numbers of super subdivision graphs $G$ in the case when $%
G $ is a bipartite graph, a complete graph or a cycle. In the next section
we consider finite super subdivision graphs of neighborhood corona graphs.
We provide the exact values for $\chi _{\rho }(FSSD_{m}(K_{n}\star P_{p}))$,
when $m$ is a positive arbitrary integer, and the upper bound for $\chi
_{\rho }(FSSD_{m}(C_{n}\star P_{p}))$. At the end, we provide some remarks
and open questions.

%The paper we continue by adding the values of the packing chromatic numbers of some corona and neighborhood corona graph, and end it with some remarks and open questions.

%%%%%%%%%%%%%%%%%%%%%%%%%%%%%%%%%%%%%%%%%%%%%%%%%%%%%%%%%%%%%%%%%%%%%%%%%%%%%%%%%%%%%%%%%%%%%%%%%%%%%%%%%%%%%%%%%%%%%%%%%%%%%%%%%%%%%%%%%%%%%%%%%%%%%%%%%%%%%%%%%%%%%%%%%%%%%%%%%%%%%%%%%%%%%%%%%%%%%%%%%%%%%%%%%%%%%%%%%%%%%%%%%%%%%%%%%%%%%%%%%%%%%%%%%%%%%%%%%%%%%%%%%%%%%%%%%%%%%%%%%%%%%%%%%%%%%%%%%%%%%%%%%%%%%%%%%%%%%%%%%%%%%%%%

\section{Notations and preliminaries}

In this paper, we consider only finite, simple graphs. For a given graph $G$%
, the vertex set of $G$ is denoted by $V(G)$ and the edge set by $E(G)$. The 
\emph{(open) neighborhood} of a vertex $u \in V(G)$ is the set of all
vertices adjacent to $u$: $N_{G}(u)=\{v\in V(G)|uv\in E(G) \}$ (we often
drop the subscript if the graph $G$ is clear from context). The \emph{degree}
of $u$, denoted by $deg_G(u)$, is $|N_G(u)|$.

The distance between two vertices $u, v \in V(G)$, denoted by $d_G(u, v)$
(or $d(u,v)$ in the case when a graph $G$ is clear from context), is the
length of a shortest $u, v$-path. The maximum of $\{d_G(x, y) : x, y \in
V(G)\}$ is called the \emph{diameter} of $G$ and is denoted by diam($G$).

Given a positive integer $i$, an \emph{$i$-packing} in $G$ is a subset $W$
of the vertex set of $G$ with the property that the distance between any two
distinct vertices from this set is greater than $i$. Note that this concept
generalizes the notion of an independent set, which is equivalent to a
1-packing. The smallest integer $k$ such that the vertex set of $G$ can be
partitioned into sets $V_1$, $V_2$, \ldots, $V_k$, where $V_i$ is an $i$%
-packing for each $i \in \{1, 2, \ldots, k \}$, is called the \emph{packing
chromatic number} of a graph $G$. We denote this number by $\chi_\rho(G)$.
The corresponding mapping $c : V(G) \rightarrow [k]$, satisfying the
property that $c(u) = c(v) = i$ implies $d_G(u,v) > i$, is called a \emph{$k$%
-packing coloring}. In the case when $k = \chi_\rho(G)$, we say that $k$%
-packing coloring is optimal.

First, recall two well known results of packing chromatic number, which will
be used several times in the sequel of this paper. While the first
proposition states that the packing chromatic number satisfies the
hereditary property, the second provides the values of packing chromatic
numbers for cycles.

\begin{proposition}
\cite{goddard-2008} For any subgraph $H$ of a given graph $G$, 
\begin{equation*}
\chi _{\rho }(H)\leq \chi _{\rho }(G).
\end{equation*}%
\label{subgraph}
\end{proposition}

\begin{proposition}
\cite{goddard-2008} If $C_{n}$ is any cycle of order $n$, then 
\begin{equation*}
\chi _{\rho }(C_{n})=\left\{ 
\begin{array}{ll}
3; & n=4k,k\geq 1,\text{ or }n=3, \\ 
4; & otherwise.%
\end{array}%
\right.
\end{equation*}%
\label{cycle_1}
\end{proposition}

Next, a \textit{finite super subdivision graph} of $G$, denoted by $%
FSSD_m(G) $, is a graph, which is obtained from $G$ by replacing each of its
edges with a complete bipartite graph $K_{2, m}$, where $m$ is a finite
number. Note that the subdivision of a given graph $G$, $S(G)$, is
equivalent to $FSSD_m(G)$, when $m=1$, and hence finite super subdivision
graphs are in some way a generalization of subdivisions of graph.

In the sequel of the paper, we use the following notations for the vertices
of \newline
$V(FSSD_m(G))$. The vertices corresponding to the vertices of $G$, denote by 
$u_1, u_2, \ldots, u_n$, and for any pair of vertices $u_i, u_j$, $i, j \in
\{1, 2, \ldots, n \}$, $i \neq j$, we denote the common neighbors of them by 
$u_{i, j}^k$ (these vertices will be called subdivided vertices), where $k
\in \{1, \ldots, m\}$.

%%%% CORONA, NEIGHBORHOOD CORONA

Next, continue with the definition of neighborhood corona graphs. 
%Given a graph $G$ with $|V(G)|=n_1$ and a graph $H$, a \textit{corona graph} of $G$ and $H$, denoted by $G \odot H$, is the graph obtained by one copy of  $G$ and $n_1$ copies of $H$, such that all vertices of $i$-th copy of $H$ are adjacent to $i$-th vertex of $G$. % \cite{k-2012}. 
%
The \textit{neighborhood corona graph} of a graph $G$ with $|V(G)|=n_1 $ and
a graph $H$ is the graph, obtained by one copy of $G$ and $n_1 $ copies of $%
H $, such that each vertex of $i$-th copy of $H$ is adjacent to all
neighbors of $i$-th vertex of $G$. This graph is denoted by $G \star H $. In
particular, when $H$ is isomorphic to $K_1$, $G \star H$ is called a \textit{%
splitting graph} and is denoted by $S^{\prime }(G)$. Note that $|V(G \star
H)|=n_1+n_1n_2$ and $|E(G \star H)|=m_{1}\left(2n_{2}+1\right)+n_{1}m_{2}$,
where $n_1$ and $n_2$ are the numbers of vertices of $G$ and $H$, and $m_1$, 
$m_2$ are the numbers of edges of $G$ resp. $H$ \cite{i-2011}.

In the sequel of this paper we consider finite super subdivision graphs of
the following neighborhood corona graphs: $K_n \star P_p$ and $C_n \star P_p$%
, where $n \geq 3$ and $p \geq 1$. We use the following notations of the
vertices of $FSSD_m(K_n \star P_p)$ respectively $FSSD_m(C_n \star P_p)$.
The vertices of $K_n \star P_p$ respectively $C_n \star P_p$, which are
corresponding to the vertices from $V(K_n)$ respectively $V(C_n)$, are
denoted by $u_1, u_2, \ldots, u_{n}$. For any $u_i, i\in \{1, \ldots, n\}$,
denote the corresponding copy of $P_p$ by $P_{i,p}$ and the vertices of $%
P_{i,p}$ by $v_{i,g}$, where $g \in \{1, 2, \ldots, p\}$ (in particular,
when $p=1$, these vertices are denoted by $v_i$). The common neighbors of $%
u_i$ and $u_j$, where $i,j \in \{1, 2, \ldots, n\}$, $i \neq j$, are denoted
by $u_{i,j}^k$, where $k \in \{1, 2, \ldots, m\}$. Analogously, the vertices
which connect $v_{i,g}$ and $v_{i,h}$, where $i \in \{1, 2, \ldots, n\}$ and 
$g, h \in \{1, 2, \ldots, p\}$, $g \neq h$, are denoted by $v_{i,g,h}^k$,
where $k \in \{1, 2, \ldots, m\}$. Finally, the vertices connecting $u_j$
and $v_{i,g}$, where $i, j \in \{1, 2, \ldots, n\}$, $i \neq j$ and $g\in
\{1, 2, \ldots, p\}$, are labeled by $s_{j,i,g}^k$, where $k \in \{1, 2,
\ldots, m\}$. In particular, if $p=1$, then the common neighbors of $u_i$
and $v_{j}$ are denoted by $s_{i,j}^k$ for any $i,j \in \{1, 2, \ldots, n\}$ 
$i \neq j$, and $k \in \{1, 2, \ldots, m\}$.

\section{Finite super subdivisions}

In this section we study general properties of finite super subdivision
graphs. While the next proposition have been already proven for subdivisions
of graph (i.e. graphs $FSSD_{m}(G)$, when $m=1$), we present its
generalization (considering the number $m$). Namely, we provide the lower
and the upper bound for $\chi _{\rho }(FSSD_{m}(G))$.

\begin{proposition}
If $m \geq 1$ and $G$ is a connected graph with at least three vertices,
then 
\begin{equation*}
\omega(G)+1 \leq \chi_\rho(FSSD_m(G)) \leq \chi_\rho(G)+1.
\end{equation*}
\label{prop1}
\end{proposition}

\begin{proof}
It is known that the result holds for $m=1$ \cite{bkr-2007}. Then, since for any $m \geq 1$ $FSSD_m(G)$ contains a subgraph isomorphic to $FSSD_1(G)$, Proposition \ref{subgraph} implies that $\omega(G)+1 \leq \chi_\rho(FSSD_1(G)) \leq \chi_\rho(FSSD_m(G))$ for each $m \geq 1$. 
In order to provide the upper bound, denote by $c$ any optimal packing coloring of $G$ and define a coloring $c'$ of $FSSD_m(G)$ as follows: $c'(u_{i,j}^k)=1$ and $c'(u_i)=c(u_i)+1$ for any $i,j \in \{1, \ldots, n\}$, $i \neq j$, and $k \in \{1, \ldots, m\}$. It is easy to check that $c'$ is a packing coloring. Namely, suppose that $c'(u_i)=c'(u_j)=l \geq 2$ for some $i$, $j$ and color $l$. Then $c(u_i)=c(u_j)=l-1$, which implies that $d_G(u_i,u_j) \geq l$ and thus $d_{FSSD_m(G)}(u_i,u_j) \geq 2l \geq l+1$ for any $l \geq 2$. Additionally, note that the distance between any two vertices both colored with color $1$ is at least $2$. Therefore, $c'$ is a packing coloring of $FSSD_m(G)$ and since it uses color $1$ and exactly $\chi_\rho(G)$ other colors, $\chi_\rho(FSSD_m(G)) \leq \chi_\rho(G)+1$, what completes the proof. 
\qed 
\end{proof}

Note that in the case of graphs $G$ with $\chi_\rho(G)=\omega(G)$, the
written bounds provide the exact values for $\chi_\rho(FSSD_m(G))$, $m \geq
1 $. For example, applying the written proposition, we derive the following
corollary.

\begin{corollary}
For any $n \geq 3$ and $m \geq 1$, $\chi_\rho(FSSD_m(K_n))=n+1.$ \label{Cor}
\end{corollary}

In addition, complete graphs (and other graphs $G$ with $\chi_\rho(G)=%
\omega(G)$) prove that the above written bounds are sharp. \newline

%%%%%%%%%%%%%%%%%%%%%%%%%%%%%%%%%%%%%%%%%%%%%%%%%%%%%%%%%%%%%%%%%%%%%%%%%%%%%%%%%%%%%%%%%%%%%%%%%%%%%%%%%%%%%%%%%%%

%%%%%%%%%%%%%%%%%%%%%%%%%%%%%%%%%%%%%%%%%%%%%%%%%%%%%%%%%%%%%%%%%%%%%%%%%%%%%
Since by Proposition \ref{prop1}, the packing chromatic number of the family
of graphs $\{FSSD_{1}(G)$,$FSSD_{2}(G), FSSD_{3}(G),\ldots \}$ is bounded
from above, and by Proposition \ref{subgraph}, $\chi_\rho(FSSD_m(G)) \leq
\chi_\rho(FSSD_{m+1}(G))$ for any graph $G$ and any $m \geq 1$, we infer
that the packing chromatic numbers of graphs $FSSD_{m}(G)$, $FSSD_{m+1}(G)$, 
$FSSD_{m+2}(G)$, \ldots ~ are equal, when $m$ is large enough. Hence, we are
interested in question, for which $m$ is $\chi_\rho(FSSD_{m}(G))=\chi_%
\rho(FSSD_{m+1}(G))=\chi_\rho(FSSD_{m+2}(G))= \ldots$. With the following
two propositions, we prove that the written equalities hold for any $m$
greater than $\frac{\chi _{\rho }(G)}{\delta (G)}$. In addition, we prove
that in the case of complete graphs of order at least $3$, bipartite graphs
of order at least $3$ and cycles, the equality actually holds for any $m
\geq 1$. %
%In other words, with the first proposition of this section is proven that the packing chromatic number of the family of graphs $\{FSSD_{1}(G),FSSD_{2}(G),\ldots \}$ is bounded, and this proposition shows which bound should $m$ achieve that the packing chromatic numbers of graphs become a constant. 

\begin{proposition}
Let $G$ be a graph and $m \geq 1$ a positive integer. If there exists an
optimal packing coloring $c$ of $FSSD_m(G)$ which assigns to all subdivided
vertices a color $1$, then $\chi_\rho(FSSD_m(G))=\chi_\rho(FSSD_{m+1}(G))$. %
\label{prop_color1}
\end{proposition}

\begin{proof} 
Let $G$ be a graph, $m \geq 1$ a positive integer and $c$ an optimal packing coloring of $FSSD_m(G)$ which assigns to all subdivided vertices a color $1$. By setting $c'(u_i)=c(u_i)$ for any $u_i$, and assigning a color $1$ to all other vertices of $FSSD_{m+1}(G)$, we get a $\chi_\rho(FSSD_m(G))$-packing coloring of $FSSD_{m+1}(G)$ and hence, $\chi_\rho(FSSD_m(G))=\chi_\rho(FSSD_{m+1}(G))$.
\qed
\end{proof}

\begin{proposition}
If $G$ is any graph and $m > \frac{\chi_\rho(G)} {\delta(G)}$ a positive
integer, then 
\begin{equation*}
\chi_\rho(FSSD_m(G))=\chi_\rho(FSSD_{m+1}(G)).
\end{equation*}
\end{proposition}

\begin{proof}
Let $m > \frac{\chi_\rho(G)}{\delta(G)}$ be an arbitrary positive integer. Then $m > \frac{\chi_\rho(G)}{deg_G(u_i)}$ for each $u_i, i \in \{1, \ldots, n\}$, and thus $m \cdot deg_G(u_i) > \chi_\rho(G)$. Clearly, $deg_{FSSD_m(G)}(u_i)=deg_G(u_i) \cdot m$ for any $u_i$, thus $1+deg_{FSSD_m(G)}(u_i) > 1+ \chi_\rho(G)$. By Proposition \ref{prop1}, it follows that $1+deg_{FSSD_m(G)}(u_i) > \chi_\rho(FSSD_m(G))$.

Next, let $c$ be any optimal packing coloring of $FSSD_m(G)$. In order to prove that if $c$ assigns a color $1$ to a vertex $u_i$, $i \in \{1, \ldots, n\}$, then it uses at least $1+deg_{FSSD_m(G)}(u_i)$ colors. Suppose that $c(u_i)=1$ for some $i \in \{1, \ldots, n\}$. Then, the neighbors of $u_i$ get pairwise different colors (since they are pairwise at distance $2$, but cannot be colored with color $1$) and hence $c$ uses at least $1+deg_{FSSD_m(G)}(u_i)$ colors (actually these colors are already required for a packing coloring of vertices from $N_{FSSD_m(G)}[u_i]$). 
But since $1+deg_{FSSD_m(G)}(u_i) > \chi_\rho(FSSD_m(G))$, $c$ is not optimal, so we have a contradiction. Therefore, $c(u_i) \neq 1$ for any $i \in \{1, \ldots, n \}$ and without loss of generality, we may assume that all other (subdivided) vertices get color $1$ (since they are pairwise at distance at least $2$). Then, Proposition \ref{prop_color1} implies the result.

%Note that, if additional color (different from colors $c(u_1), c(u_2), \ldots, c(u_n)$) is not required for packing coloring of the vertices $u_{i,j}^k$, $i, j \in \{1, \ldots, n\}$, $i \neq j$, $k \in \{1, \ldots, m\}$, then one of the colors $c(u_1), c(u_2), \ldots, c(u_n)$ is $1$, a contradiction. 

%\Jasmina{In addition, the argument also holds for any optimal packing coloring of $FSSD_{l}(G)$ for $l > m$}.
%Note that $d_{FSSD_m(G)}(u_i, u_j)= d_{FSSD_{m+1}(G)}(u_i, u_j)$ for any two $u_i, u_j$, where $i,j \in \{1, \ldots, n\}$, $i \neq j$. Thus any optimal packing coloring of $FSSD_m(G)$ and $FSSD_{m+1}(G)$ use the same number of colors for the packing coloring of vertices from $\{u_1, \ldots, u_n\}$ (and all of the others vertices get color $1$). Therefore $\chi_\rho(FSSD_m(G))=\chi_\rho(FSSD_{m+1}(G))$.
\qed 
\end{proof}

Therefore, if $m > \frac{\chi_\rho(G)}{\delta(G)}$, then $%
\chi_\rho(FSSD_m(G))=\chi_\rho(FSSD_{m+1}(G))$, but the case when $m \leq 
\frac{\chi_\rho(G)}{\delta(G)}$ is still opened. Clearly, also in this case $%
\chi_\rho(FSSD_m(G)) \leq \chi_\rho(FSSD_{m+1}(G))$ holds for any graph $G$
and any $m \geq 1$. As we will see, there exist graphs, for example cycles
of order $2k+1$, $k \geq 2$ (see Propositions \ref{cycle_1} and \ref{cycle}%
), which satisfy the property that $\chi_\rho(G)=\chi_\rho(FSSD_m(G))$ for
any $m \geq 1$. But for the others, there is a question of when there appear
a change of the value of the packing chromatic number. More precisely, for
which $m$ we have $\chi_\rho(FSSD_m(G)) <
\chi_\rho(FSSD_{m+1}(G)=\chi_\rho(FSSD_{m+i}(G))$ for any $i \geq 2$? Based
on the Proposition \ref{prop_color1}, we derive that a necessary condition
for the inequality of $\chi_\rho(FSSD_m(G))$ and $\chi_\rho(FSSD_{m+1}(G))$
is that any optimal packing coloring of $FSSD_m(G)$ assigns to at least one
of the subdivided vertices a color greater than $1$. For example, any
optimal packing coloring of $FSSD_1(K_2)$ assigns to a subdivided vertex a
color $2$ and we have: $2=\chi_\rho(FSSD_1(K_2))<\chi_\rho(FSSD_{m}(K_2))=3$%
, $m \geq 2$. Also for the Petersen graph, $\chi_\rho(FSSD_1(P))=5$ (see 
\cite{bkrw-2017b}) and it is easy to observe that there exists an optimal
packing coloring $FSSD_m(P)$, $m \geq 2$, which assigns to all subdivided
vertices a color $1$, which implies that $\chi_\rho(FSSD_m(P))=\chi_%
\rho(FSSD_{m+1}(P)) \geq 6$ for any $m \geq 2$. We have to mention that we
were not able to find any other graph $G$ such that $\chi_\rho(FSSD_1(G))<%
\chi_\rho(FSSD_{2}(G))$, and in addition, we have not found any graphs $G$
with the property that $\chi_\rho(FSSD_m(G))<\chi_\rho(FSSD_{m+1}(G))$ for
any $m \geq 2$. Hence, there arises an open question of whether there exists
a graph $G$ with the written property. \newline

%%%%%%%%%%%%%%%%%%%%%%%%%%%%%%%%%%%%%%%%%%%%%%%%%%%%%%%%%%%%%%%%%%%%%%%%%%%%%%%%%
We continue with the consideration of packing chromatic numbers of finite
super subdivision graphs of bipartite graphs and (other) cycles. While the
result has been already known for all bipartite graphs (it follows from
Proposition 3.3. in \cite{goddard-2008}), we give the exact values also for
cycles. \newline
%%%%%%%%%%%%%%%%%%%%%%%%%%%%%%%%%%%%%%%%%%%%%%%%%%%%%%%%%%%%%%%%%%%%%%%%%%%%%%%%%%%%%%%%%%%%%%%%%%%%%%%%%%%%%%%%%%%%%%%%%%%%%%%%%%%%%%%%%%%%%%%%

\begin{proposition}
For any bipartite graph $G$ of order at least $3$ and any $m\geq 1$, 
\begin{equation*}
\chi _{\rho }(FSSD_{m}(G))=3.
\end{equation*}%
\label{bipartite}
\end{proposition}

%%%%%%%%%%%%%%%%%%%%%%%%%%%%%%%%%%%%%%%%%%%%%%%%%%%%%%%%%%%%%%%%%%%%%%%%%%%%%%%%%%%%%%%%%%%%%%%%%%%%%%%%%%%%%%%%%%%%%%%%%%%%%%%%%%%%%%%%%%%%%%%%%%%%%%%%%%%%%%%%%%%%%%%%%%%%%%%%%%%%%%%%%%%%%%%%%%%%%%%%%%%%%%%%%%%%%%%%%%%%%%

\begin{proposition}
If $n\geq 3$ and $m\geq 1$, then 
\begin{equation*}
\chi _{\rho }(FSSD_{m}(C_{n}))=\left\{ 
\begin{array}{ll}
3; & n~\text{is even}, \\ 
4; & n~\text{is odd}.%
\end{array}%
\right.
\end{equation*}%
\label{cycle}
\end{proposition}

\begin{proof}
Let $n \geq 3$ and $m \geq 1$ be arbitrary integers. If $n$ is even, then Proposition \ref{bipartite} implies the result. Otherwise, note that $FSSD_m(C_n)$ contains a subgraph, which is isomorphic to a cycle $C_{2n}$. Since in this case $2n$ is not a multiple of $4$ and $2n \neq 3$, the results from Proposition \ref{subgraph} and Proposition \ref{cycle_1} imply that $\chi_\rho(FSSD_{m}(C_{n})) \geq 4$. To show that $\chi_\rho(FSSD_{m}(C_{n})) \leq 4$, color all vertices $u_{i, j}^k$, $i, j \in \{1, \ldots, n\}$, $i \neq j$, $k \in \{1, \ldots, m \}$, with color $1$ and the consecutive vertices $u_1, u_2, \ldots, u_n$ one after another using the following pattern of colors: $2, 3, 2, 3, \ldots, 2, 3, 4$. Clearly, this is a $4$-packing coloring of $FSSD_m(C_n)$, thus $\chi_\rho(FSSD_{m}(C_{n})) = 4$.
\qed
\end{proof}

%%%%%%%%%%%%%%%%%%%%%%%%%%%%%%%%%%%%%%%%%%%%%%%%%%%%%%%%%%%%%%%%%%%%%%%%%%%%%%%%%%%%%%%%%%%%%%%%%%%%%%%%%%%%%%%%%%%%%%%%%%%%%%%%%%%%%%%%%%%%%%%%%%%%%%%%%%%%%%%%%%%%%%%%%%%%%%%%%%%%%%%%%%%%%%%%%%%%%%%%%%%%%%%%%%%%%%%%%%%%%%%%%%%%%%%%%%%%%%%%%%%%%%%%%%%%%%%%%%%%%%%%%%%%%%%%%%%%%%%%%%%%%%%%%%%%%%%%%%%%%%%%%%%%%%%%%%%%%%%%%%%%%%
%%%%%%%%%%%%%%%%%%%%%%%%%%%%%%%%%%%%%%%%%%%%%%%%%%%%%%%%%%%%%%%%%%%%%%%%%%%%%%%%%%%%%%%%%%%%%%%%%%%%%%%%%%%%%%%%%%%%%%%%%%%%%%%%%%%%%%%%%%%%%%%%%%%%%%%%%%%%%%%%%%%%%%%%%%%%%%%%%%%%%%%%%%%%%%%%%%%%%%%%%%%%%%%%%%%%%%%%%%%%%%%%%%%%%%%%%%%%%%%%%%%%%%%%%%%%%%%%%%%%%%%%%%%%%%%%%%%%%%%%%%%%%%%%%%%%%%%%%%%%%%%%%%%%%%%%%%%%%%%

\section{Finite super subdivision graphs of neighborhood corona graphs}

We continue with determining the packing chromatic numbers of finite super
subdivision graphs of neighborhood corona graphs. Note that neighborhood
corona graphs were defined in Section 2.

Before determining $\chi_\rho(FSSD_m(K_n \star P_p))$ we need the following
lemma. %%%%%%%%%%%%%%%%%%%%%

\begin{lemma}
Let $n\geq 3$, $p\geq 2$ and $m\geq 1$ be arbitrary integers. Then, for any $%
a$-packing coloring $c$ of the graph $FSSD_{m}(K_{n}\star P_{p})$, where $a
\leq n+3$, the following holds:

\begin{enumerate}
\item $c(u_i) \neq 1$ for all $i \in \{1, 2, \ldots, n\}$;

\item $c(v_{i,g}) \neq 1$ for all $i \in \{1, 2, \ldots, n\}$ and $g \in
\{1, 2, \ldots, p\}$.
\end{enumerate}

\label{lemma_color1}
\end{lemma}

\begin{proof}
Let $n\geq 3$, $p\geq 2$, $m\geq 1$ be arbitrary integers, and let $c$ be an arbitrary $a$-packing coloring of $FSSD_{m}(K_{n}\star P_{p})$, where $a \leq n+3$.

First we prove that $c(u_i) \neq 1$ for all $i \in \{1, 2, \ldots, n\}$. Suppose to the contrary that $c(u_i)=1$ for some $i \in \{1, \ldots, n\}$ and without loss of generality, assume that $i=1$. 
Note that $u_1$ has $m(n-1)+m(n-1)p$ neighbors. Since $m \geq 1$ and $p \geq 2$, we derive that $u_1$ has actually at least $3n-3$ neighbors, which are pairwise at distance $2$ and hence $c$ assigns them pairwise different colors from $\{2, 3, \ldots \}$. Therefore, $c$ uses at least $3n-3+1=3n-2$ colors, what is more than $a$, a contradiction to $c$ being an $a$-packing coloring. Hence, $c(u_i) \neq 1$ for all $i \in \{1, \ldots, n\}$ and without loss of generality, we may assume that $c(u_{i,j}^k)=1$ for all $i, j \in \{1, \ldots, n\}$, $i \neq j$, and $k \in \{1, \ldots, m\}$.  

% 2. del 
Next, we  prove that $c(v_{i,g}) \neq 1$ for all $i \in \{1, 2, \ldots, n\}$ and $g \in \{1, 2, \ldots, p\}$. Again, suppose to the contrary that $c(v_{i,g}) = 1$ for some $i$ and $g$, say $c(v_{2,1}) = 1$. Note that $v_{2,1}$ has at least $mn$ neighbors. If $m \geq 2$, then $deg(v_{2,1}) \geq 2n$ and hence $c$ uses at least $2n+1$ different colors, which is more than $a$, a contradiction. 
Therefore, in the remainder of the proof, we only need to consider the case when $m=1$.
We distinguish three cases with respect to $n$. 

\textbf{Case $1$.} $n \geq 5$. \\
In this case, $c$ assigns to (at least) four vertices $s_{i,2,1}^{1}$, $i \in \{1, 3, 4, \ldots, n\}$ pairwise different colors from $\{2, 3, 4, \ldots\}$, and note that at least three of the mentioned vertices receive colors from  $\{3, 4, \ldots \}$ by $c$. Since these colors cannot be used for packing coloring of the vertices from $\{u_1, u_2, \ldots, u_n\}$, Corollary~\ref{Cor} implies that $c$ uses at least $n+4$ colors, a contradiction. 

\textbf{Case 2.} $n=4$. \\
Vertices $s_{1,2,1}^{1}$, $s_{3,2,1}^{1}$, $s_{4,2,1}^{1}$ and $v_{2,1,2}^{1}$ receive four different colors by $c$. Note that, if  $c(s_{1,2,1}^{1})$, $c(s_{3,2,1}^{1})$, $c(s_{4,2,1}^{1}) \in \{3, 4, \ldots \}$, then by the same consideration as above follows a contradiction. Therefore, $c(s_{i,2,1}^{1})=2$ for some $i \in \{1,3,4\}$, say $i=1$. 
If $c(v_{2,1,2}^{1})=3$, then by using Corollary~\ref{Cor}, we infer that $c$ uses at least $n+3$ colors for a packing coloring of the subgraph of $FSSD_{1}(K_{n}\star P_{p})$, which is induced by the set of vertices $V(FFSD_1(K_4)) \cup \{s_{1,2,1}^{1}, s_{3,2,1}^{1}, s_{4,2,1}^{1}, v_{2,1}, v_{2,1,2}^{1} \}$. Since $a \leq n+3$, the colors 2 and 3 must be used for a packing coloring of the vertices from $\{u_1, u_2, u_3, u_4\}$, which implies $c(u_2)=3$. 
But then it is easy to check that there is no available colors for at least one of the vertices $v_{2,2}$, $s_{1,2,2}^{1}$, $s_{3,2,2}^{1}$ or $s_{4,2,2}^{1}$ (note that at least one of them cannot receive a color $1$). The same result follows if $c(v_{2,1,2}^{1})=4$. In the case when $c(v_{2,1,2}^{1}) \geq 5$, colors $c(v_{2,1,2}^{1})$, $c(s_{3,2,1}^{1})$ and $c(s_{4,2,1}^{1})$ cannot be used by $c$ for packing coloring of the vertices of subgraph isomorphic to $FSSD_1(K_4)$, thus by Corollary~\ref{Cor}, $c$ uses more than $a$ colors, a contradiction. 

\textbf{Case 3.} $n=3$. \\
Recall that $c$ is an $a$-packing coloring of $FSSD_{1}(K_{3}\star P_{p})$ and in this case $a \leq 6$. 
Since $c(u_i) \neq 1$ for all $i \in \{1,2,3\}$, let be $c(u_1)=b$, $c(u_2)=c$ and $c(u_3)=d$, where $b,c,d$ are three pairwise distinct colors from $\{2,3, \ldots a\}$. 
We consider five sub-cases with respect to the colors $b,c,d$.

%First we show that in each situation $c$ uses $6$ colors for the packing coloring of a subgraph induced by the set of vertices $u_1,u_2,u_3, u_{1,2}, u_{1,3}, u_{2,3}, v_{3,1}$, $v_{3,2}, v_{3,1,2}^{1}$, $s_{1,3,1}^{1}$, $s_{1,3,2}^{1}$, $s_{2,3,1}^{1}$, $s_{2,3,2}^{1}$ and then we prove that $c$ uses more than $7$ colors for a packing coloring of the whole graph. 

\textbf{Case 3.1.}
$b=2.$ \\
Vertex $s_{1,2,1}^{1}$ can be colored only with color $e \in \{3,4, \ldots, a\}$ by $c$, where $e \neq c,d$, and  $s_{3,2,1}^{1}$ can receive one of the colors from $\{2, f\}$ by $c$, where $f \in \{3,4, \ldots, a\}$, $f \neq c,d,e$. 

If $c(s_{3,2,1}^{1})=2$, then $c(v_{2,1,2}^{1}) \in \{f, c\}$. The first case implies that $c(v_{2,2})=c=3$. 
The second case yields that $c \leq 4$ and $c(v_{2,2}) \in \{1, f\}$; if $1$ is used, then $c(s_{1,2,2}^{1})=f$ and $c(s_{3,2,2}^{1})=e=3$.

If $c(s_{3,2,1}^{1})=f$, then $c(v_{2,2}) \in \{1,c\}$; the first case implies that $f=3$ ($c(s_{1,2,2}^{1})=f=3)$, and the second that $c=3$. 

Next, consider the vertices $s_{1,3,1}^{1}$, $s_{1,3,2}^{1}$, $v_{3,1}$ and $v_{3,2}$. It is easy to observe that in each case there is no available colors by $c$ for at least one of the mentioned vertices, a contradiction to $c$ being an $a$-packing coloring.

\textbf{Case 3.2.} $c=2$ \\
In this case $\{c(s_{1,2,1}^{1}), c(s_{3,2,1}^{1}), c(v_{2,1,2}^{1})\} = \{2, e, f\}$, $e,f \in \{3,  \ldots, a\}$, $e \neq f$, $e \neq b, d$, $f \neq b,d$.

First, suppose that $c(v_{2,1,2}^{1})=2$ (and $c(s_{1,2,1}^{1})=e$, $c(s_{3,2,1}^{1})=f$). This implies that $c(v_{2,2})=1$ and $c(s_{1,2,2}^{1})=f=3$, but then there is no available color for the vertex $s_{3,2,2}^{1}$, which yields a contradiction to $c$ being an $a$-packing coloring. 
Therefore $c(v_{2,1,2}^{1}) \neq 2$ and without loss of generality we may assume that $c(s_{1,2,1}^{1})=2$ (and $c(s_{3,2,1}^{1})=e$, $c(v_{2,1,2}^{1})=f$). Next distinguish three posibilities with respect to the colors $c(v_{1,1})$ and $c(v_{1,2})$.

\textbf{Case 3.2.1.} $c(v_{1,1}), c(v_{1,2}) \neq 1$. \\
The only possibility is that the vertices $v_{1,1}$ and $v_{1,2}$ receive colors $3=b$ and $4=f$ by $c$, which implies $c(v_{3,1})=c(v_{3,2})=1$. But then there is no available colors for $s_{2,3,2}^{1}$, a contradiction to $c$ being an $a$-packing coloring.

\textbf{Case 3.2.2.} $c(v_{1,1})=1, c(v_{1,2}) \neq 1$ (the proof in the case when $c(v_{1,1}) \neq 1, c(v_{1,2}) = 1$ is analogous). \\
Vertex $s_{2,1,1}^{1}$ can receive color $e$, if $e=3$, or color $f$, if $f \in \{3,4,5\}$, by $c$.

If $c(s_{2,1,1}^{1})=3$ (i.e. either $e=3$ or $f=3$), then $c(v_{1,2})=f=4$ (and therefore $c(s_{2,1,1}^{1})=e=3$), so $v_{3,1}$ and $v_{3,2}$ can receive only color $1$ by $c$, but then we have the same contradiction as in situation a. The case when $c(s_{2,1,1}^{1})=4=f$ yields $c(v_{1,2})=3=b$ and again, the vertices $v_{3,1}$ and $v_{3,2}$ can receive only color $1$ by $c$ (note that $d,e \geq 5$), a contradiction. Therefore, $c(s_{2,1,1}^{1})=5=f$ and then $c(v_{1,2})=3=b$, which yields that the vertices $v_{3,1}$ and $v_{3,2}$ can receive only colors $1$ and $e=4$ by $c$. In each case it is impossible to color all of the vertices from $N(v_{3,1}) \cup N(v_{3,2})$ by $c$, a contradiction to $c$ being an $a$-packing coloring.

\textbf{Case 3.2.3.} $c(v_{1,1})=c(v_{1,2})=1$. \\
This implies that $c(s_{3,1,1}^{1})$ is $2$ or $f=3$. If $c(s_{3,1,1}^{1})=f=3$, then there is no available colors for $s_{2,1,1}^{1}$. Thus $c(s_{3,1,1}^{1})=2$, which implies $c(s_{3,1,2}^{1})=f=3$, but then there is no available colors for $s_{2,1,2}^{1}$, a contradiction to $c$ being an $a$-packing coloring.

\textbf{Case 3.3.} $b=3$ and $c,d \neq 2$ (i.e. $c,d \in \{4, \ldots, a\}$). \\
In this case $\{c(s_{1,2,1}^{1}), c(s_{3,2,1}^{1})\} = \{2,e\}$, where $e \in \{4, \ldots, a\}, e \neq c,d$. First, consider the case when $c(s_{1,2,1}^{1})=2$ and $c(s_{3,2,1}^{1})=e$. We derive that $c(s_{1,3,1}^{1})=c(s_{1,3,2}^{1})=1$, hence the vertices $v_{3,1}$ and $v_{3,2}$ receive color $2$ and color $e=4$ (thus, $c \neq 4$). But then, there is no available color for $v_{2,1,2}^{1}$, a contradiction. 
If $c(s_{1,2,1}^{1})=e$ and $c(s_{3,2,1}^{1})=2$, then the vertices $s_{1,3,1}^{1}$, $s_{1,3,2}^{1}$, $s_{2,3,1}^{1}$, $s_{2,3,2}^{1}$, $v_{3,1}$ and $v_{3,2}$ receive the colors from $\{1,2\}$. But since it is impossible to color all of the mentioned vertices with these two colors, we have a contradiction to $c$ being a packing coloring. 

\textbf{Case 3.4.} $c=3$ and $b,d \neq 2$ (i.e. $b,d \in \{4, \ldots, a\})$. \\
This yields that $\{c(s_{1,2,1}^{1}), c(s_{3,2,1}^{1})\}=\{2,e\}$, where $e \in \{4, \ldots, a\}, e \neq b,d$.
First assume that $c(s_{1,2,1}^{1})=2$ and $c(s_{3,2,1}^{1})=e$. The vertices $s_{3,1,1}^{1}$, $s_{2,1,1}^{1}$, $s_{2,1,2}^{1}$, $s_{3,1,2}^{1}$, $v_{1,1}$ and $v_{1,2}$ receive colors $1$ and $2$. Again, it is impossible to color all of the mentioned vertices only with two colors, thus we have a contradiction. 
If $c(s_{1,2,1}^{1})=e$ and $c(s_{3,2,1}^{1})=2$, then by considering the vertices $s_{2,3,1}^{1}$, $s_{1,3,1}^{1}$, $s_{2,3,2}^{1}$, $s_{1,3,2}^{1}$, $v_{3,1}$ and $v_{3,2}$ by analogous consideration as above follows a contradiction. 

\textbf{Case 3.5.} $\{b,c,d\} = \{4,5,6\}$. \\
Vertices $s_{1,2,1}^{1}$ and $s_{3,2,1}^{1}$ can receive only colors $2$ and $3$, thus $c(v_{2,1,2}^{1})=c=4$. Without loss of generality, we may assume that $c(s_{1,2,1}^{1})=2$ and $c(s_{3,2,1}^{1})=3$. Then the vertices $s_{3,1,1}^{1}$ and $s_{3,1,2}^{1}$ can get only colors $1$ or $2$, so at least one of them is colored by $1$, say $s_{3,1,1}^{1}$. This yields that $c(v_{1,1})=2$, $c(v_{1,2})=1$ and $c(v_{1,1,2}^{1})=3$, but then there is no available colors for $s_{2,1,2}^{1}$ and $s_{3,1,2}^{1}$ (only one of them can get color $2$), a contradiction to $c$ being an $a$-packing coloring.  

By symmetry, we can switch the role of the colors $b$ and $d$, and derive that our claim holds.
\qed
\end{proof}

\begin{theorem}
If $n\geq 3$, $p\geq 2$ and $m\geq 1$, then 
\begin{equation*}
\chi_\rho(FSSD_{m}(K_{n}\star P_{p}))=n+3.
\end{equation*}
\label{theorem_kn}
\end{theorem}

\begin{proof}
In order to show that $\chi_\rho(FSSD_{m}(K_{n}\star P_{p})) \leq n+3$ holds for any $n\geq 3$, $p\geq 2$ and $m\geq 1$, we form a $(n+3)$-packing coloring of $FSSD_{m}(K_{n}\star P_{p})$. 
First, color the vertices $u_{i,j}^k$, $v_{i, g, h}^k$ and $s_{i, j, g}^k$ for all $i, j \in \{1, 2, \ldots, n\}$, $i \neq j$,  $g, h \in \{1, 2, \ldots, p\}$, $g \neq h$ and $k \in \{1, \ldots, m\}$, with color $1$.
Then color the uncolored vertices, which correspond to the vertices from $V(P_{i,p})$, $i \in \{1, \ldots, n\}$, one after another with the following pattern of colors: $2, 3, 2, 3, \ldots$ 
Finally, color the vertices $u_1, u_2, \ldots, u_n$ with $n$ new colors.  
It is easy to check that the described coloring is $(n+3)$-packing coloring of a given graph, thus $\chi_\rho(FSSD_{m}(K_{n}\star P_{p})) \leq n+3$.

Next, prove that $\chi_\rho(FSSD_{m}(K_{n}\star P_{p})) \geq n+3$ holds for any $n\geq 3$, $p\geq 2$ and $m\geq 1$. First, consider a graph $FSSD_m(K_{n}\star P_{2})$, where $n \geq 3$ and $m \geq 1$ are arbitrary integers. Denote by $c$ an optimal packing coloring of this graph. Note that $c$ is an $a$-packing coloring, where $a \leq n+3$. 
Using Lemma \ref{lemma_color1} we infer that $c(u_i) \neq 1$ and $c(v_{i,g}) \neq 1$ for all $i \in \{1, 2, \ldots, n\}$ and $g \in \{1, 2\}$.
Let be $A=\{u_i; 1 \leq i \leq n\} \cup \{v_{i,g}; 1 \leq i \leq n, 1 \leq g \leq 2\}$. Note that all vertices from $V(FSSD_m(K_{n}\star P_{2})) \setminus A$ (can) receive color $1$ by $c$ and all vertices from $A$ receive colors from $\{2, 3, \ldots, a\}$ by $c$. Since the vertices from $A$ are pairwise at distance at most $4$, we have: $|c^{-1}(i) \cap A|=1$ for all $i \in \{4, 5, \ldots, a\}$. In other words, for each color $i \in \{4, 5, \ldots, a\}$, there is only one vertex from $A$, colored by $i$. Then, since $c$ is an optimal packing coloring of $FSSD_m(K_{n}\star P_{2})$, it assigns colors $2$ and $3$ to at most as possible vertices. 
If there exists $i \in \{1, \ldots, n\}$ such that $c(u_i)=2$ (respectively $c(u_i)=3$), then $|c^{-1}(2) \cap A| \leq 2$ (respectively  $|c^{-1}(3) \cap A| \leq 2 $).
Otherwise, $c$ can assign a color $2$ (respectively $3$) to at most $n$ vertices (one vertex in each $P_2$ can be colored with color $2$), which is more than $2$. 
Therefore, $c(u_i) \neq 2$ and $c(u_i) \neq 3$ (and recall that $c(u_i) \neq 1$) for all $i \in \{1, 2, \ldots, n\}$, which implies that $c$ uses at least $n+3$ colors. Hence, $\chi_\rho(FSSD_{m}(K_{n}\star P_{2})) \geq n+3$. Furthermore, since $FSSD_{m}(K_{n}\star P_{p})$, $m \geq 1, n \geq 3, p \geq 2$, contains a subgraph isomorphic to $FSSD_{m}(K_{n}\star P_{2})$, $\chi_\rho(FSSD_{m}(K_{n}\star P_{p})) \geq \chi_\rho(FSSD_{m}(K_{n}\star P_{2})) \geq n+3$, what completes the proof. 
\qed
\end{proof}

In Fig. \ref{fig:8} is shown a graph $FSSD_{m}(K_{3}\star P_{2})$ and its
packing coloring, as is described in the proof of the previous theorem. Note
that all unlabeled vertices of presented graph receive a color $1$.

\begin{figure}[h]
\centering
\includegraphics[scale=1]{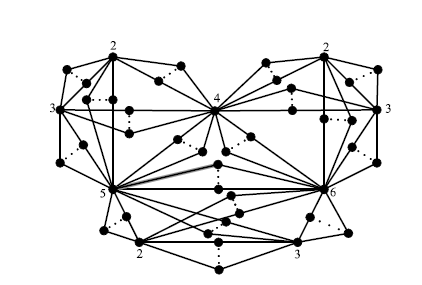} 
\caption{$FSSD_m(K_3 \star P_2)$ and a packing coloring of this graph}
\label{fig:8}
\end{figure}

%%%%%%%%%%%%%%%%%%%%%%%%%%%%%%%%%%%%%%%%%%%%%%%%%%%%%%%%%%%%%%%%%%%%%%%%%%%%%%%%%%%%%%%%%%%%%%%%%%%%%%%%%%%%%%%%%%%%%%%%%%%%%%%%%%%%%%%%%%%%%%%%%%%%%%%%%%%%%%%%%%%%%%%%%%%%%%%%%%%%%%%%%%%%%%%%%%%%%%%%%%%%%%%%%%%%%%%%%%%%%%%%%%%%%%%%%%%%%%%%%%%%%%%%%%%%%%%%%%%%%%
We continue this section with determining the packing chromatic numbers of
graphs $FSSD_{m}(C_{n}\star P_{p})$, $n\geq 3,p\geq 2$. While in the case of
graphs $FSSD_{m}(K_{n}\star P_{p})$, we provided the exact values of their
packing chromatic numbers, in the case of cycles the task gets much harder
for us, hence we present only the upper bound for $\chi _{\rho
}(FSSD_{m}(C_{n}\star P_{p}))$.

\begin{theorem}
If $n\geq 3$, $m\geq 1$ and $p\geq 2$, then

\begin{equation*}
\chi _{\rho }(FSSD_{m}(C_{n}\star P_{p}))\leq \left\{ 
\begin{array}{ll}
6; & n=3, \\ 
7; & n\geq 4,n\notin \{5,7,11\}, \\ 
8; & n\in \{5,7,11\}.%
\end{array}%
\right.
\end{equation*}%
\label{TH_cycle}
\end{theorem}

\begin{proof}
In the case when $n=3$ the results clearly holds, since 
$C_3$ is isomorphic to $K_3$ and thus $FSSD_m(K_3 \star P_p)$ is isomorphic to $FSSD_m(C_3 \star P_p)$ for any $m \geq 1$ and $p \geq 2$. Thus, Theorem \ref{theorem_kn} yields the result.

In order to prove the desired bounds in the case when $n \geq 4$, we form a packing coloring $c$ of a given graph $FSSD_{m}(C_{n} \star P_{p})$.
First, let be $c(u_{i,j}^k)=c(v_{i,g,h}^k)=c(s_{j,i,g}^k)=1$ for all $i,j \in \{1, \ldots, n\}$, $i \neq j$, $g, h\in \{1, \ldots, p\}$, $g \neq h$, $k \in \{1, \ldots, m\}$. Then, for each $i \in [n]$ color the vertices from $\{v_{i,g}$; $1 \leq g \leq p\}$ one after another with the following sequence of colors: $2, 3, 2, 3, \ldots$. The remaining vertices of $G$ (i.e. the vertices $u_i$) color one after another using the following pattern of colors.

\textbf{Case 1.} $n=4, n=5$. \\
In the case when $n=4$ use the pattern $4567$, and in the  case when $n=5$, color the vertices $u_i$ with color pattern $45678$.

\textbf{Case 2.} $n \cong 0$ (mod 6). \\
Use the color pattern $456 457$ for the packing coloring of the vertices $u_i$. 

\textbf{Case 3.} $n \cong 1$ (mod 6). \\
In this case start the coloring of the consecutive vertices with the colors $754 657 456 7456$ and repeat the pattern $457 456$. If $n=13$, then the repeated block is omitted, and if $n=7$, then assign to the vertices $u_i, i \in [7]$, the colors of the following pattern: $4564578$. 

\textbf{Case 4.} $n \cong 2$ (mod 6). \\
Color the vertices one after another using this pattern of colors $754 657 456 74564$ and repeat the pattern $754 654$, if necessary. If $n=8$, then use the color pattern: $75467456$.

\textbf{Case 5.} $n \cong 3$ (mod 6). \\
When $n \cong 3 $(mod 6), start the coloring of the vertices $u_i$ with the colors $4657$, repeat the block $456 457$ and end by $456 75$. Note that, the repeated block is omitted in the case when $n=9$ (see Fig. \ref{fig:3}). 
 
\textbf{Case 6.} $n \cong 4$ (mod 6). \\
In this case repeat the sequence of colors $456 457$ and end by $4567$. 

\textbf{Case 7.} $n \cong 5$ (mod 6). \\
Start the coloring with the patten $754 657 456$, repeat the block $457 456$ and  end by $754 67546$. Note that the repeated block can be omitted. In the case when $n=11$, use the colors $75465745648$.

Since in each case the described coloring is a packing coloring of a given graph, the proof is completed.
\qed
\end{proof}

\begin{figure}[h]
\centering
\includegraphics[scale=1]{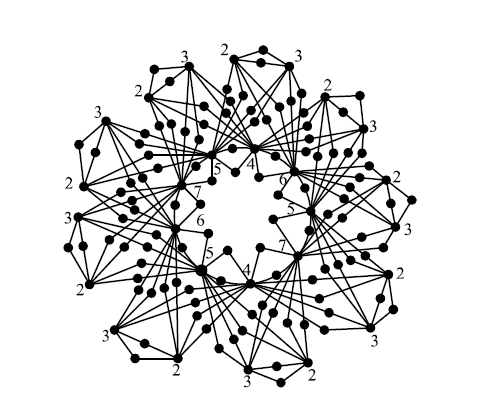} 
\caption{$FSSD_2(C_9 \star P_2)$ and a packing coloring of this graph}
\label{fig:3}
\end{figure}

In the sequel of this section, we determine the packing chromatic numbers of
graphs $FSSD_{m}(K_{n}\star P_{p})$ and $FSSD_{m}(C_{n}\star P_{p})$, when $%
p=1$. In other words, we consider so called splitting graph.

\begin{proposition}
If $n\geq 3$ and $m\geq 1$, then 
\begin{equation*}
\chi _{\rho }(FSSD_{m}(S^{^{\prime }}(C_{n}))=\left\{ 
\begin{array}{ll}
3; & \text{if $n$ is even}, \\ 
5; & \text{if $n$ is odd}.%
\end{array}%
\right.
\end{equation*}%
\label{cycle'}
\end{proposition}

%%%%%%%%%%%%%%%%%%%%%%%%%%%%%%%%%%%%%%%%%%%%%%%%%%%%%%%%%%%%%%%%%%%%%%%%

\begin{proof} 

\textbf{Case 1.} $n$ is even.\\ 
Since $FSSD_{m}(S^{\prime }(C_{n}))$ contains a subgraph, which is isomorphic to $FSSD_m(C_{n})$, by Propositions \ref{subgraph} and \ref{cycle} follows that $\chi_\rho(FSSD_{m}(S^{\prime }(C_{n})) \geq 3$ for any $m \geq 1$ and $n \geq 3$. 

In order to prove that $\chi_\rho(FSSD_{m}(S^{^{\prime }}(C_{n})) \leq 3$ holds for any $m\geq 1$ and any even $n \geq 3$, we form a $3$-packing coloring of considered graph. First, color all vertices $u_i$, $i \in \{1, \ldots, n \}$, one after another using the following pattern of colors: $2, 3, 2, 3, \ldots, 2, 3$. 
Then, color each vertex $v_i$, $1 \leq i \leq n$, with the color, which is assigned to $u_i$ (for example, see Fig. \ref{fig:4}). The remaining vertices are colored with color $1$. Clearly, this is a $3$-packing coloring of considered graph and therefore $\chi_\rho(FSSD_{m}(S^{^{\prime }}(C_{n})) = 3$ for any $m \geq 1$ and any even $n \geq 3$.

\textbf{Case 2.} $n$ is odd. \\
A graph $FSSD_{m}(S^{^{\prime }}(C_{n}))$ contains a subgraph, which is isomorphic to $FSSD_m(C_{n})$, and analogically as above by Propositions \ref{subgraph} and \ref{cycle} follows that $\chi_\rho(FSSD_{m}(S^{\prime }(C_{n}))) \geq 4$ for any $m \geq 1$ and any odd $n \geq 3$. 

Next, suppose that $\chi_\rho(FSSD_{m}(S^{^{\prime }}(C_{n}))) = 4$ and let $c$ be any $4$-packing coloring of considered graph.
If $c(u_i)=1$ for some $i \in \{1, \ldots, n\}$, then all neighbors of $u_i$ get pairwise different colors. Since $deg(u_i) \geq 4$, $c$ uses at least $5$ colors, a contradiction. Therefore, $c(u_i) \in \{2, 3, 4\}$ for all $i \in \{1, \ldots, n\}$. 
If there exists $i \in \{1, \ldots, n\}$, such that $c(u_i)=4$ and $c(u_{i-1})=2$, $c(u_{i+1})=3$ (resp., $c(u_{i-1})=3$, $c(u_{i+1})=2$),  then there is no available color for $v_i$ or its neighbor.
Therefore, $c(u_i)=4$ implies that either $c(u_{i-1})=2$ and $c(u_{i+1})=2$ or $c(u_{i-1})=3$ and $c(u_{i+1})=3$. But then, by replacing color $4$ with color $3$ (for each $u_i$, when $c(u_i)=4$ and $c(u_{i-1})=c(u_{i+1})=2$) or color $2$ (for each $u_i$, when $c(u_i)=4$ and $c(u_{i-1})=c(u_{i+1})=3$), we infer that $\chi_\rho(FSSD_m(C_{n})) \leq 3$, a contradiction to Proposition \ref{cycle}. Thus, $\chi_\rho(FSSD_{m}(S^{^{\prime }}(C_{n})) \geq 5$.

In order to show that $\chi_\rho(FSSD_{m}(S^{^{\prime }}(C_{n}))) \leq 5$, we form a $5$-packing coloring of considered graph. First, color all vertices $u_i$, $1 \leq i \leq n-1$, using the following pattern of colors: $2,3,2,3, \ldots, 2,3$. Then color each vertex $v_i$, $1 \leq i \leq n-1$, with the color assigned to $u_i$, vertex $u_n$ with color $4$ and vertex $v_n$ with color $5$.  The remaining vertices color by $1$. This is a $5$-packing coloring of a given graph and hence $\chi_\rho(FSSD_{m}(S^{^{\prime }}(C_{n})) =5$.
\qed
\end{proof}

Fig. \ref{fig:4} shows a graph $FSSD_m(S^{\prime }(C_4))$ and its packing
coloring, described in the previous proof.

\begin{figure}[h]
\centering
\includegraphics[scale=1]{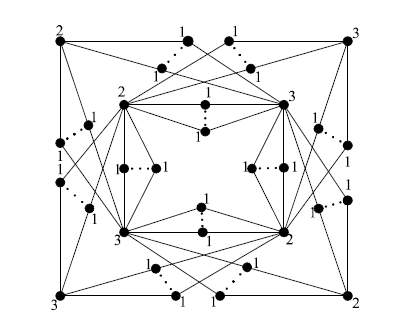} 
\caption{$FSSD_m(S^{\prime }(C_4))$ and a packing coloring of this graph}
\label{fig:4}
\end{figure}
%%%%%%%%%%%%%%%%%%%%%%%%%%%%%%%%%%%%%%%%%%%%%%%%%%%%%%%%%%%%%%%%%
%%%%%%%%%%%%%%%%%%%%%%%%%%%%%%%%%%%%%%% 

\begin{proposition}
If $n\geq 3$ and $m\geq 1$, then 
\begin{equation*}
\chi_\rho(FSSD_{m}(S^{\prime }(K_{n})))=n+2.
\end{equation*}
\end{proposition}

\begin{proof} 

In order to show that $\chi_\rho(FSSD_{m}(S^{\prime }(K_{n}))) \leq n+2$ holds for any $n \geq 3$ and $m \geq 1$, form a $(n+2)$-packing coloring of $FSSD_m(S^{\prime }(K_{n}))$. First, color all vertices $v_i$, $ 1 \leq i \leq n$, with color $2$, and all vertices $u_i$, $1 \leq i \leq n$, with pairwise different colors from  $\{3, 4, \ldots, n+2\}$. Finally, the remaining vertices of a given graph color with color $1$. Clearly, such coloring is $(n+2)$-packing coloring of $FSSD_{m}(S^{\prime }(K_{n}))$ and thus $\chi_\rho(FSSD_{m}(S^{\prime }(K_{n}))) \leq n+2$ for all $n \geq 3$, $m \geq 1$.

Next, prove that $\chi_\rho(FSSD_{m}(S^{\prime }(K_{n}))) \geq n+2$ holds for any $n \geq 3$ and $m \geq 1$. If $n=3$, then Proposition \ref{cycle'} implies that $\chi_\rho(FSSD_{m}(S^{\prime }(K_{3})))=5$ and we are done. 
Otherwise, let $c$ be any optimal packing coloring of $FSSD_{m}(S^{\prime }(K_{n}))$, where $n \geq 4$. Note that $c$ uses at most $n+2$ colors.
Suppose that there exists $i \in \{1, \ldots, n\}$ such that $c(u_i)=1$. Then, $c(u_{i,j}^k) \neq 1$ and $c(s_{i,j}^k) \neq 1$ for any $j \in \{1, \ldots, n\}$, $i \neq j$, and $k \in \{1, \ldots, m\}$. Since $m \geq 1$, $c$ uses at least $(n-1)+(n-1)$ colors different from $1$, hence $\chi_\rho(FSSD_m(S^{\prime }(K_{n})))\geq 2n-1$. Since $2n-1 > n+2$ for any $n \geq 4$, we have a contradiction to $c$ being an optimal packing coloring of a given graph. Therefore, $c(u_i) \neq 1$ for any $i \in \{1, \ldots, n\}$. 
Then, suppose that there exists $j \in \{1, \ldots, n\}$ such that $c(v_j)=1$. Without loss of generality assume that $c(v_1)=1$, which implies that $c(s_{i,1}^k) \neq 1$ for any $i \in \{2,3, \ldots, n\}$. Recall that also $c(u_i) \neq 1$ for any $i \in \{1, \ldots, n\}$. 
Since the vertices from $\{u_1, u_{i}, s_{i,1}^k; 2 \leq i \leq k\}$ are pairwise at distance at most $3$, the only color which can be used at least twice for packing coloring of these vertices, is color $2$. 
But since any two vertices $s_{i,1}^k$, $i \in \{2, \ldots, n\}$, and also any two vertices $u_i$, $i \in \{1, \ldots, n\}$ are pairwise at distance $2$, $c$ assigns a color $2$ to at most two of the mentioned vertices. Therefore, $c$ uses at least $(n-1)+(n-1)+1$ colors (at least $n-1$ colors for the neighbors of $v_1$, beside that also $n-1$ additional colors for the vertices $u_i$ and color $1$), what is more than $n+2$ for any $n \geq 4$, a contradiction. Hence, $c(v_i) \neq 1$ for any $i \in  \{1, \ldots, n\}$. 
Without loss of generality we may assume that $c$ assigns to all vertices from $V(FSSD_m(S'(K_n))) \setminus \{u_i, v_i; 1 \leq i \leq n \}$ a color $1$.
Note that any two vertices from $A=\{u_i, v_i; i=1, \ldots, n\}$ are at distance at most $4$, which implies $|c^{-1}(i) \cap A| \leq 1$ for all $i \geq 4$.
If $c$ assigns colors $2$ and $3$ to two vertices from $\{u_i; 1, \ldots, n\}$, then at most one vertex from $\{v_i; 1, \ldots, n\}$ receive color $2$ and at most one receive a color $3$. Using Corollary \ref{Cor} and the fact that $n \geq 4$ we infer that $c$ uses at least $n+3$ colors, a contradiction. If $c(u_i) \neq 2$ or $c(u_i) \neq 3$ for all $i \in \{1, \ldots, n\}$, then all vertices $v_i$ (can) receive a color $2$ respectively $3$, which yields that $c$ uses at least $n+2$ colors ($n+1$ color for the vertices $u_1, u_2, \ldots, u_n$ and a color $2$ resp. $3$). This completes the proof.
\qed
\end{proof}

%%%%%%%%%%%%%%%%%%%%%%%%%%%%%%%%%%%%%%%%%%%%%%%%%%%%%%%%%%%%%%%%%%%%%%%%%%%%%%%%%%%%%%%%%%%%%%%%%
%%%%%%%%%%%%%%%%%%%%%%%%%%%%%%%%%%%%%%%%%%%%%%%%%%%%%%%%%%%%%%%
%%%%%%%%%%%%%%%%%%%%%%%%%%%%%%%%%%%%%%%%%%%%%%%%%%%%%%%%%%%%%%

\section{Concluding remarks}

It is well known that some operations or only local changes of a given
graph, change its packing chromatic number. While there are known some fact
about the packing chromatic number of a subdivision of a given graph, in
this paper we consider the operation of finite super subdivisions. Since
there are known some results relating subdivisions of graphs and the concept
of so called $S=(s_1,s_2, \ldots, s_k)$-packing coloring (see e.g. \cite%
{gt-2016}), it would be interesting to consider such coloring and the
influence of the operation of finite super subdivisions on it.

There are some additional open problems about finite super subdivisions of
graphs that follows directly from our work. Namely, as we mentioned, we have
found only two graphs $G$ such that $\chi_\rho(FSSD_1(G))<\chi_%
\rho(FSSD_{2}(G))$. Therefore, an open problem is to determine all graphs $G$
such that $\chi_\rho(FSSD_1(G))<\chi_\rho(FSSD_{2}(G))$. In addition, we
have not found any graph $G$ with the property that $\chi_\rho(FSSD_m(G))<%
\chi_\rho(FSSD_{m+1}(G))$ for any $m \geq 2$, and it would be interesting to
know whether there exists any such graph $G$.

Another natural problem that arises from Theorem \ref{TH_cycle}, is to
determine the exact values of packing chromatic numbers of graphs $%
FSSD_m(C_n \star P_p)$. We propose also a problem of providing the packing
chromatic number of finite super subdivisions of some other classes of
graphs (not necessarily neighborhood corona graphs).

\section*{Acknowledgements}

J.F. acknowledges the financial support from the Slovenian Research Agency
(P1-0403, J1-9109 and J1-1693).

%%%%%%%%%%%%%%%%%%%%%%%%%%%%%%%%%%%%%%%%%%

\end{document}